\documentclass[10pt,a4paper]{article}
\usepackage[utf8]{inputenc}
\usepackage[english]{babel}
\usepackage{hyperref}
\usepackage{amsmath,amsthm,enumerate}
\usepackage{amssymb}
\usepackage[capitalize]{cleveref}
\usepackage{tikz}\usetikzlibrary{calc}\usetikzlibrary{positioning}
\usepackage{subfigure}
\usepackage[hmargin=2.5cm,vmargin=3cm]{geometry}
\usepackage[color=green!30]{todonotes}

\usepackage{perpage} 
\MakePerPage{footnote} 

\newtheorem{theorem}{Theorem}

\newtheorem{proposition}[theorem]{Proposition}

\newtheorem{claim}{Claim}[theorem]

\newcommand{\tw}{\text{tww}} 

\newcommand{\claimproof}{\noindent\emph{Proof of claim.} }
\newcommand{\smallqed}{{\tiny ($\Box$)}\medskip}

\begin{document}

\title{Neighbourhood complexity of graphs of bounded twin-width\footnote{Florent Foucaud was financed by the French government IDEX-ISITE initiative 16-IDEX-0001 (CAP 20-25) and by the ANR project GRALMECO (ANR-21-CE48-0004). Tuomo Lehtil\"a's research was supported by the Finnish Cultural Foundation and by the Academy of Finland grant 338797.}}
\author{\'Edouard Bonnet\footnote{\noindent Univ Lyon, CNRS, ENS de Lyon, Université Claude Bernard Lyon 1, LIP UMR5668, France.}
\and Florent Foucaud\footnote{\noindent Université Clermont Auvergne, CNRS, Clermont Auvergne INP, Mines Saint-Etienne, LIMOS, 63000 Clermont-Ferrand, France.}~~\footnote{\noindent Univ. Orléans, INSA Centre Val de Loire, LIFO EA 4022, F-45067 Orléans Cedex 2, France.}
\and Tuomo Lehtil\"a\footnote{Univ Lyon, UCBL, CNRS, LIRIS - UMR 5205, F69622, France.}~~\footnote{University of Turku, Department of Mathematics and Statistics, Turku, Finland.}
\and Aline Parreau\footnote{\noindent Univ Lyon, CNRS, INSA Lyon, UCBL, Centrale Lyon, Univ Lyon 2, LIRIS, UMR5205, F-69622 Villeurbanne, France.}
}

\maketitle

\begin{abstract}
We give essentially tight bounds for, $\nu(d,k)$, the maximum number of distinct neighbourhoods on a set $X$ of $k$ vertices in a graph with twin-width at most~$d$.
 Using the celebrated Marcus-Tardos theorem, two independent works [Bonnet et al., Algorithmica '22; Przybyszewski '22] have shown the upper bound $\nu(d,k) \leqslant \exp(\exp(O(d)))k$, with a double-exponential dependence in the twin-width. The work of [Gajarsky et al., ICALP '22], using the framework of local types, implies the existence of a single-exponential bound (without explicitly stating such a bound).  
 We give such an explicit bound, and prove that it is essentially tight. Indeed, we give a short self-contained proof that for every $d$ and $k$ $$\nu(d,k) \leqslant (d+2)2^{d+1}k = 2^{d+\log d+\Theta(1)}k,$$ and build a bipartite graph implying $\nu(d,k) \geqslant 2^{d+\log d+\Theta(1)}k$, in the regime when $k$ is large enough compared to~$d$.
\end{abstract}

\section{Introduction}\label{sec:intro}

The aim of this paper is to refine our understanding of how complex the neighbourhoods of graphs of bounded twin-width can be. We provide an improved bound on the neighbourhood complexity of such graphs, complemented by a construction showing that our bound is essentially tight. The improvements in the bounds for neighbourhood complexities translate directly to better structural bounds and algorithms, in some contexts which are explained below.

\paragraph{Twin-width.} Twin-width is a recently introduced graph invariant~\cite{twin-width-I}; see~\cref{sec:prelim} for a definition.
It can be naturally extended to matrices over finite alphabets and binary structures~\cite{twin-width-I,twin-width-IV,twin-width&permutations}.
Although classes of bounded twin-width are broad and diverse, they allow (most of the time, provided a~witness is given as an input) improved algorithms, compared to what is possible on general graphs or binary structures.

Most prominently, it was shown~\cite{twin-width-I} that, on $n$-vertex graphs given with a~$d$-sequence (a witness that their twin-width is at most~$d$), deciding if a~first-order sentence $\varphi$ holds can be solved in time $f(d,\varphi) n$, for some computable function $f$.
In some special cases, such as for \textsc{$k$-Independent Set} or \textsc{$k$-Dominating Set}\footnote{That is, the problems of deciding whether in an input graph, there are $k$ vertices that are pairwise non-adjacent or whose closed neighbourhood is the entire vertex set, respectively.}, single-exponential parameterised algorithms running in time $2^{O_d(k)} n$ are possible~\cite{twin-width-III}.
In the same setting, the triangles of an $n$-vertex $m$-edge graph can be counted in time $O(d^2 n+m)$~\cite{Kratsch22}.
See~\cite{twin-width-VI,Gajarsky22,PilipczukSZ22} for more applications of twin-width with an algorithmic flavour.

Classes of binary structures with bounded twin-width include bounded treewidth, and more generally, bounded clique-width classes, proper minor-closed classes, posets of bounded width (that is, whose antichains are of bounded size), hereditary subclasses of permutations, as well as $\Omega(\log n)$-subdivisions of $n$-vertex graphs~\cite{twin-width-I}, and particular classes of (bounded-degree) expanders~\cite{twin-width-II}.
A~rich range of geometric graph classes have bounded twin-width such as map graphs, bounded-degree string graphs~\cite{twin-width-I}, classes with bounded queue number or bounded stack number~\cite{twin-width-II}, segment graphs with no $K_{t,t}$ subgraph, and visibility graphs of simple polygons without large independent sets~\cite{twin-width-VIII}, to give a few examples.

If efficiently approximating the twin-width is a challenging open question in general,  this is known to be possible for the above-mentioned classes (albeit a~representation may be needed for the geometric classes) and for ordered graphs~\cite{twin-width-IV}.
By that, we mean that there are two computable functions $f, g$ and an algorithm that, for an input $n$-vertex graph $G$ from the class and an integer $k$, in time $g(k) n^{O(1)}$, either outputs an $f(k)$-sequence (again, witnessing that the twin-width is at most~$f(k)$) or correctly reports that the twin-width of $G$ is larger than~$k$. 

Structural properties of graph classes of bounded twin-width include $\chi$-boundedness~\cite{twin-width-III}, with a polynomial binding function~\cite{bourneuf2023bounded}, smallness (i.e., containing up to isomorphism $2^{O(n)}$ $n$-vertex graphs)~\cite{twin-width-II,twin-width&permutations}, and Vapnik-Chervonenkis (VC) density at most~1~\cite{twin-width-kernels,Gajarsky22,twin-width-VCdensity}. The latter property is the topic of the current article.

\paragraph{VC density and neighbourhood complexity.}
VC density is related to the celebrated VC dimension~\cite{VCdim}.
Given a~set-system (or hypergraph) $\mathcal S$ on a~domain $X$, the \emph{shatter function} $\pi_{\mathcal S}: \mathbb N \to \mathbb N$ is defined as $$\pi_{\mathcal S}(n)=\max\limits_{A \in {X \choose n}} |\{Y \subseteq A~|~\exists S \in \mathcal S, Y = A \cap S\}|.$$
The Perles-Sauer-Shelah lemma states that $\pi_{\mathcal S}(n)=O(n^d)$ if the VC dimension of $\mathcal S$ (i.e., the supremum of $\{n~|~\pi_{\mathcal S}(n)=2^n\}$) is a finite integer~$d$.
Then the VC density of $\mathcal S$ is defined as $\inf \{c \in \mathbb R~|~\pi_{\mathcal S}(n)=O(n^c)\}$, and as $+\infty$ if the VC dimension is unbounded.

We define the \emph{VC density} of an infinite class $\mathcal C$ of finite graphs as the VC density of the infinite set-system formed by the neighbourhood hypergraph of the disjoint union of the graphs of~$\mathcal C$, that is, $\{N_G(v)~|~v \in V(\uplus_{G \in \mathcal C} G)\}$, where $N_G(v)$ denotes the set of neighbours of $v$ in $G$.
The VC density is an important measure in finite model theory, often more tractable than the VC dimension (see for instance~\cite{Aschenbrenner13,Aschenbrenner16}).
Tight bounds have been obtained for the VC density of (logically) definable hypergraphs from graph classes of bounded clique-width~\cite{Paszke20} (with monadic second-order logic), and more recently, of bounded twin-width \cite{Gajarsky22} (with first-order logic).

In structural graph theory and kernelisation~\cite{Fomin19} (a subarea of parameterised complexity~\cite{Cygan15}) the function $\pi_{\mathcal N(G)}$, where $\mathcal N(G)$ is the neighbourhood hypergraph of $G$, is often\footnote{Some authors define the neighbourhood complexity as $n \mapsto \frac{\pi_{\mathcal N(G)}(n)}{n}$.} called \emph{neighbourhood complexity}. (See~\cite{BFS19} for an algorithmic study of the computation of this notion.)
In these contexts, obtaining the best possible upper bound for $\pi_{\mathcal N(G)}$ (and not just the exponent matching the VC density) translates to qualitatively better structural bounds and algorithms; see for instance~\cite{twin-width-kernels,reduced-bandwidth,Eickmeyer17,Reidl19}.

The \emph{$r$-neighbourhood complexity} of $G$ is the neighbourhood complexity of $G^r$, with same vertex set as $G$, and an edge between two vertices at distance at most~$r$ in $G$.
Reidl et al.~\cite{Reidl19} showed that among subgraph-closed classes, bounded expansion\footnote{A~notion from the Sparsity theory of Ne\v{s}et\v{r}il and Ossona de Mendez~\cite{Sparsity} extending bounded degree and proper minor-free classes.} is equivalent to linear $r$-neighbourhood complexity.
Indeed, the more general nowhere dense classes~\cite{nowhere}\footnote{Another invention of the Sparsity program~\cite{Sparsity}.} have almost linear $r$-neighbourhood complexity~\cite{Eickmeyer17}: there is a function $f: \mathbb N \times \mathbb N \to \mathbb N$ such that for every $\varepsilon > 0$, $\pi_{\mathcal N(G^r)}(n) \leqslant f(r,\varepsilon) n^{1+\varepsilon}$ for all $n$.
On hereditary classes, i.e., closed under taking \emph{induced} subgraphs, there is no known characterisation of linear neighbourhood complexity.

As we already mentioned in a different language, bounded twin-width classes have been proven to have linear neighbourhood complexity. See~\cite[Lemma 3]{twin-width-kernels} or~\cite[Section 3]{twin-width-VCdensity} for two independent proofs, both using the Marcus-Tardos theorem~\cite{MarcusT04}.
However, the dependence in the twin-width is doubly exponential in both papers.
Setting $\nu(d,k)$ as the maximum number of distinct neighbourhoods on a set of size $k$ within a graph of twin-width at most~$d$, i.e., $\max\{\pi_{\mathcal N(G)}(k)~|~G~\text{has twin-width at most}~d\}$, they show that $\nu(d,k) \leqslant \exp(\exp(O(d)))k$.
There is a~recent third proof not using the Marcus-Tardos theorem~\cite{Gajarsky22}.
The authors tackle the~more general problem of bounding the number of distinct first-order definable subsets within a fixed set. 
In the particular case of neighbourhoods, even though this is not made explicit in~\cite{Gajarsky22}, their proof gives a similar upper bound of $\nu(d,k)$ to ours.

\paragraph{Our results.} In this note, we give in~\cref{sec:upperbound} a short and self-contained proof (also not using the Marcus-Tardos theorem) that $\nu(d,k) \leqslant 2^{d+\log d+\Theta(1)}k$.
In~\cref{sec:lowerbound}, we complement that proof with a~construction of a bipartite graph witnessing that $\nu(d,k) \geqslant 2^{d+\log d+\Theta(1)}k$, which makes our single-exponential upper bound in twin-width essentially tight.

\section{Preliminaries}\label{sec:prelim}

We use the standard graph-theoretic notations: $V(G)$, $E(G)$, $G[S]$, $G - S$ respectively denote the vertex set, edge set, subgraph of $G$ induced by $S$, and subgraph of $G$ induced by $V(G) \setminus S$.
If $v \in V(G)$, then the \textit{open neighbourhood of vertex $v$ in $G$} denoted by $N_G(v)$ (or $N(v)$ if $G$ is clear from the context) is the set of neighbours of $v$ in $G$. 
If $X \subseteq V(G)$, then an \emph{$X$-neighbourhood} is a set $N(v) \cap X$ for some $v \in V(G)$.

We now define the twin-width of a graph, following the definition of \cite{twin-width-I}. 

A {\em trigraph} is a triple $G=(V(G),E(G),R(G))$ where $E(G)$ and $R(G)$ are two disjoint sets of edges on $V(G)$: the usual edges (also called {\em black edges}) and the {\em red edges}. Informally, a red edge between two vertices $u$ and $v$ means that some errors have been made between $u$ and $v$. The {\em red degree} of a trigraph is the maximum degree of the graph $(V(G),R(G))$. Any graph $G$ can be interepreted as a trigraph $G=(V(G),E(G),\emptyset)$. 
Given a trigraph and two vertices $u,v\in V(G)$ (not necessarily adjacent), the trigraph $G/u,v = G'$ is obtained by {\em contracting} $u$ and $v$ in a new vertex $w$ such that:
\begin{itemize}
    \item $V(G')=\{w\}\cup V(G)\setminus \{u,v\}$;
    \item the edges between vertices of $V(G)\setminus \{u,v\}$ are the same in $G'$;
    \item we set the edges incident to $w$ in the following way:
    \begin{itemize}
        \item $wx\in E(G')$ if $xu\in E(G)$ and $xv \in E(G)$;
        \item $wx\notin E(G')\cup R(G')$ if $xu\notin E(G)\cup R(G)$ and $xv \notin E(G)\cup R(G)$;
        \item $wx\in R(G')$ otherwise.
    \end{itemize}
\end{itemize}

In other words, the common black neighbours of $u$ and $v$ are black neighbours of $w$. All the other neighbours of $u$ or $v$ are red neighbours of $w$. Red edges stay red, black edges stay black, red and black edges become red. Moreover, non-edges stay as non-edges, non-edges and red edges become red edges, and non-edges and black edges become red edges.
We say that $G/u,v$ is a \emph{contraction} of $G$.
A~\emph{$d$-sequence} of an $n$-vertex graph $G$ is a sequence of $n$ trigraphs $G=G_n, G_{n-1},....,G_1$ such that each trigraph $G_i$ is obtained from $G_{i+1}$ by a contraction and has red degree at most~$d$. 
The {\em twin-width} of $G$, denoted by $\tw(G)$, is the minimum integer~$d$ such that $G$ admits a $d$-sequence. 
Note that an induced subgraph of~$G$ has a twin-width smaller or equal to the twin-width of~$G$~\cite{twin-width-I}.

If $u \in G_i$, then $u(G)$ denotes the set of vertices of $G$ eventually contracted to $u$ in $G_i$. 
Instead of considering the trigraphs $G_i$, we might prefer to deal with the partitions of $V(G)$ induced by the sets $u(G)$ for $u$ in $G_i$:
 $\mathcal P_i=\{u(G)~|~u \in V(G_i)\}$. In this setting, we say that $u(G)$ is a \textit{part} of $\mathcal{P}_i$. 
 We say that there is a red edge, a black edge or a non-edge between two parts $u(G)$ and $v(G)$ of $\mathcal P_i$ if $uv$ is a red edge, a black edge or a non-edge in $G_i$.

\section{Upper bound on the number of distinct neighbourhoods}\label{sec:upperbound}

We state and prove our upper bound on the maximum number of distinct $X$-neighbourhoods in bounded twin-width graphs.

\begin{theorem}\label{The:twinwidth}
Let $G$ be an $n$-vertex graph of twin-width~$d$, and $X \subseteq V(G)$ where $X\neq\emptyset$. Then the number of distinct $X$-neighbourhoods in $G$ is at most $(d+2)2^{d+1}|X| = 2^{d+\log d+\Theta(1)}|X|$.
\end{theorem}
\begin{proof}
Fix non-empty $X \subseteq V(G)$. 
First of all, for all vertices of $V(G) \setminus X$ with the same $X$-neighbourhood, we keep only one representative. 
Note that the new graph $G''$ is an induced subgraph of $G$, thus its twin-width is at most $d$. We further modify graph $G''$ by adding for each $v\in X$ a new vertex $u$ to $G''$ so that $N(u)=N(v)$ if such vertex does not exist in $V(G'')\setminus X$. We do this one vertex at a time. The new graph is called $G'$ and it has the same twin-width as $G''$.

Let $M=(d+2)2^{d+1}+1$. 
We prove by induction on $n$ that an $n$-vertex graph of twin-width at most $d$ with a set $X$ of $k\geq1$ vertices, where all vertices outside $X$ have a distinct $X$-neighbourhood, satisfies $n \leqslant kM$. 
This will prove that $G'$ has at most $kM$ vertices, and thus that in $G$, there are at most $(M-1)k$ distinct $X$-neighbourhoods.

The statement is trivially true for $n \leqslant 5$ since $M \geqslant 5$, for all $d \geqslant 0$.

Thus, assume $n \geqslant 6$. 
In particular, we have $k>1$. Let $x\in X$. Let $X'=X\setminus \{x\}$ and let $T_x$ be the set of pairs of vertices outside $X$ that are twins with respect to $X'$, i.e. $$T_x=\left\{\{u,v\} \in {{V(G') \setminus X} \choose 2} \mid N(u)\cap X'=N(v)\cap X'\right\}.$$ 
Since every vertex of~$V(G') \setminus X$ has a~distinct neighbourhood in $X$, there are at most two vertices of $V(G') \setminus X$ with the same (possibly empty) neighbourhood $N$ in $X'$; namely the vertices $u, v \in V(G') \setminus X$ with $N(u) \cap X = N$ and $N(v) \cap X = N \cup \{x\}$ (if they exist).
Hence, $T_x$ consists of pairwise-disjoint pairs of vertices.

We prove the following claim.

\begin{claim}\label{clm}
There exists a vertex $x$ of $X$ such that $T_x$ comprises at most $M-1$ pairs, in $G'$.
\end{claim}
\claimproof 
Consider a $d$-sequence of contractions $G'_n,\ldots,G'_1$ of $G'$. Consider the last step  $G'_i$ of the sequence where all the parts of $\mathcal P_i$ contain at most one vertex of $X$ (that is, contrary to $\mathcal P_i$, some part of $\mathcal P_{i-1}$ contains two vertices of $X$).

Let $P$ be a part of $\mathcal P_i$. Let $x$ be the unique (if there exists one) element of $P \cap X$. Then we claim that $|P\setminus X|\leqslant 2^{d+1}$. Indeed, any two vertices of $P\setminus X$ have some vertex in the symmetric difference of their $X$-neighbourhoods, either it is $x$, or some vertex $x'$ of $X$ outside $P$. If that distinguishing vertex is some $x'$ that is not in $P$, then there has to be a red edge between $P$ and the part that contains $x'$. 
There are at most $d$ red edges with $P$ as an extremity. Since all the elements of $X$ are in distinct parts in $G'_i$, it means that $d+1$ vertices of $X$ are enough to distinguish all the $X$-neighbourhoods of vertices of $P\setminus X$, and thus $|P\setminus X|\leqslant 2^{d+1}$.

We now consider the next contraction in the sequence, which leads to $G'_{i-1}$. By definition of $G'_i$, it must contract two vertices corresponding to two parts of $\mathcal P_i$ that both contain an element of $X$. Let $x_1$ and $x_2$ be these two elements of $X$. Let $Q$ be the part of $\mathcal P_{i-1}$ that contains both $x_1$ and $x_2$.
Let $\{u,v\}$ be a pair of $T_{x_1}$ and let $T_{x_1}$ contain $M'$ pairs. Since $u$ and $v$ have the same neighbourhood in $X\setminus \{x_1\}$, it means that they are either both adjacent or both non-adjacent to $x_2$, and exactly one of them is adjacent to $x_1$. Thus, necessarily, one vertex among the pair $\{u,v\}$ is adjacent to exactly one vertex among $\{x_1,x_2\}$. In particular, if this vertex is not in $Q$, then there has to be a red edge between the part containing this vertex and the part $Q$ in $G'_{i-1}$.
Since $T_{x_1}$ contains $M'$ pairs (which are disjoint) and $Q$ has at most $2^{d+2}$ vertices not in $X$, there are at least $M'-2^{d+2}$ vertices not in $X$ whose part in $G'_{i-1}$ has a red edge to $Q$. 
Since each other part has at most $2^{d+1}$ vertices not in $X$, it makes at least $\frac{M'-2^{d+2}}{2^{d+1}}$ red edges incident to $Q$.
Thus, we must have 
$\frac{M'-2^{d+2}}{2^{d+1}}\leqslant d$, leading to $M'\leqslant 2^{d+1}(d+2)=M-1$, which proves the claim.~\smallqed

By Claim~\ref{clm}, there exists a vertex $x\in X$ such that $|T_x|\leqslant M-1$.
Let $Y$ be a set of $|T_x|$ vertices that intersects each pair of $T_x$ exactly once.
Let $G_Y = G' - (Y \cup \{x\})$. Then, $X'=X\setminus \{x\}$ is a vertex set of size $k-1$ such that all $X'$-neighbourhoods of vertices outside $X'$ are distinct. 
The graph $G_Y$ has at least $n-M$ vertices, and twin-width at most $d$. 
By induction, we have $n-M\leqslant|V(G_Y)|\leqslant (k-1)M$ and thus, $n\leqslant kM$. Hence, once we recall that no vertex in $X$ has unique $X$-neighbourhood, there are at most $(M-1)k$ distinct $X$-neighbourhoods, which completes the proof.
\end{proof}

\section{Lower bound on the number of distinct neighbourhoods}\label{sec:lowerbound}

Notice that when $|X|$ and $\tw(G)$ are roughly the same, the bound from Theorem~\ref{The:twinwidth} cannot be sharp, since $G'$ has at most $2^{|X|}+|X|$ vertices. However, when $|X|$ is large enough compared to $\tw(G)$, we next show that the bound is sharp up to a constant factor.

\begin{proposition}\label{prop:tww-example}
There is a positive constant $c$, such that for any integer $d$, there is a bipartite graph $G$ of twin-width at most $d$, and a large enough set $X \subseteq V(G)$, with at least $c \cdot d2^d |X|=2^{d+\log d+\Theta(1)}|X|$ distinct $X$-neighbourhoods in $G$.
\end{proposition}

\begin{proof}
Observe that the claim is clearly true for any small $d$. Thus, we do not need to consider separately graphs with small twin-width upper bounded by a constant. Hence, we assume from now on that $d\geq d'$ where $d'$ is some positive constant (at least $3$).

We construct the graph $G$ as follows.
Let $A$, $B$, $C\in \mathbb{Z}$ be three constants that will be given later ($A$ and $B$ will be roughly equal to $\sqrt{d}$ and $C$ will be roughly equal to $d$). Let $X=\{x_1,...,x_k\}$ be an independent set of $k\geq d+2\sqrt{d-2}+1$ 
vertices. 
Our goal is to construct $G$ so that each vertex 
%
in $V(G)\setminus X$ has a unique $X$-neighbourhood. 
For any integers $i,j,t$ with $1\leqslant i\leqslant  j \leqslant i+ A-1$, $j+2\leqslant t \leqslant j+1+B$ and $t\leqslant k-C$, we create a set $V_{i,j,t}$ of vertices as follows. Consider the set $X_t=\{x_{t+1},...,x_{t+C}\}$. For every subset $Y$ of $X_t$, let $Y'=\{x_i,...,x_j,x_t\}\cup Y$ and add a vertex $v_{Y'}$ to $V_{i,j,t}$, making it adjacent to the vertices of $Y'$.
Each set $V_{i,j,t}$ has size $2^{C}$ and there are $\Theta(kAB)$ (for fixed $A$, $B$ and $C$, and growing $k$) such sets. 
Thus there are $\Theta(kAB2^{C})$ vertices in the graph.

Any two vertices not in $X$ have distinct $X$-neighbourhoods. Indeed, by considering the natural ordering of $X$ induced by the indices, any vertex not in $X$ is first adjacent to a consecutive interval of vertices from $x_i$ to $x_j$, then is not adjacent to vertices from $x_{j+1}$ to $x_{t-1}$ (which is not empty since $t\geqslant j+2$), and then adjacent to $x_t$. Thus, if two vertices have the same $X$-neighbourhood, they must be in the same set $V_{i,j,t}$. But then, they have a distinct neighbourhood in $\{x_{t+1},...,x_{t+C}\}$.

We now prove that the twin-width of $G$ is at most $M=\max\{AB,C\}+2$. For that, we give a sequence of contractions with red degree at most $M$.

The contraction sequence is split into $k-C$ steps. During these steps we first consider  vertices of $X$ one by one and then in the last one we deal with the remaining vertices of $X$.
Let $0\leq \ell \leq k-C-1$. Step~$0$ corresponds to the starting point, where each vertex is alone. Let $\ell\geqslant 1$. 
After Step~$\ell$, there will be the following parts in the corresponding partition (vertices not in any of the mentioned parts are in corresponding singleton parts containing only the vertices themselves):
\begin{itemize}
    \item Let $i=\ell$. For each $j,t$ such that $i\leqslant j \leqslant i+A-1$ and $j+2\leqslant t\leqslant j+1+B$, there is a part $B_{j,t}$. The parts $B_{i,t}$ (parts with $j=i$), contain all the vertices of the sets $V_{i',j',t}$ such that $j'\leq i$. The parts $B_{j,t}$ with $j>i$ contain all the vertices of the sets $V_{i',j',t}$ such that $i'\leqslant i$ and $j'=j$. Note that there are $AB$ non-empty $B_{j,t}$ parts in total.
    \item There is a part $X_0$ that contains vertices from $x_1$ to $x_\ell$ of $X$.
    \item There is a part $T$ (for ``trash'') that contains all the vertices of the sets $V_{i',j,t}$ with $t\leqslant \ell+1$.
\end{itemize}

All the other vertices are not yet contracted. This corresponds to the vertices from $x_{\ell+1}$ to $x_k$ of $X$ and to the vertices of the sets $V_{i',j,t}$ with $i'>i=\ell$. Indeed, if $i'\leqslant i$ and $t\leqslant i+1$, then the vertices of $V_{i',j,t}$ are in $T$. If $t\geqslant i+2$ but $j\leqslant i$, then they are in the part $B_{i,t}$. If $j>i$, then they are in the part $B_{j,t}$.

We first prove that the red degree after Step~$\ell$ is at most $M$. Then, we explain how to get from Step $\ell$ to Step~$\ell+1$ by keeping the red degree at most $M$.

Consider the part $B_{j,t}$ at the end of Step~$\ell$. A vertex in this part belongs to some set $V_{i',j',t}$ with $i'\leqslant i=\ell$ and $j'=j$ if $j>i$ or $j'\leqslant i$ otherwise. In particular, two vertices of $B_{j,t}$ are adjacent to all the vertices between $x_{i+1}$ and $x_j$, to no vertex between $x_{j+1}$ and $x_{t-1}$, to $x_t$, and to no vertex after $x_{t+C}$. Thus, there is a red edge between the parts $B_{j,t}$ and $X_0$, and $C$ red edges between the part $B_{j,t}$ and the vertices $\{x_{t+1},...,x_{t+C}\}$. Therefore, the number of red edges incident with $B_{j,t}$ is at most $C+1$.

Consider now the part $T$. Vertices in $T$ are adjacent only to vertices of $X$ up to $x_{\ell+C+1}$. Since vertices $x_1$ to $x_\ell$ are all in the part $X_0$, the red degree of $T$ is at most $C+2$.

Single vertices not in $X$ have no incident red edges: indeed, they are all in some sets $V_{i',j,t}$ for $i'>i=\ell$ and thus are not adjacent to any vertex of $X_0$. For the same reason, there are red edges incident to $X_0$ only to $T$ and to the parts $B_{j,t}$. Hence, the red degree of $X_0$ is at most $AB+1$. Similarly, the red degree of $x_{i'}$, $i'>i+1$ is at most $AB+1$. Moreover, the red degree of $x_{i+1}$ is at most one. Indeed, the only red edge is between $x_{i+1}$ and $T$.

Finally, the red degree after step~$\ell$ is at most $\max \{AB+1,C+2\}\leqslant M$.

Let $\ell\geq 0$. We now explain how we perform the contractions to go from step~$\ell$ to step~$\ell+1$.
\begin{enumerate}
\item (only if $\ell\geq 1$) Let $i=\ell$. For any $i+3\leqslant t\leqslant i+2+B$, merge the part $B_{i,t}$ with the part $B_{i+1,t}$ resulting in part $B_{i+1,t}$. 
The only new red edge this merging may lead to, when $B_{i,t}$ is non-empty, is between $B_{i+1,t}$ and $x_{i+1}$.
Thus, we add only one red edge between $x_{i+1}$ and $B_{i+1,t}$. Thus, the red degree of $B_{i+1,t}$ is at most $C+2$ and the red degree of $x_{i+1}$ is at most $2$.
\item Add all the vertices of $V_{i+1,j,t}$ for some $j,t$ to the part (that might be empty at this point) $B_{j,t}$. The red degree of $B_{j,t}$ is at most $C+2$ since we might have a red edge between $B_{j,t}$ and $x_{i+1}$. The number of nonempty parts $B_{j,t}$ at this point is at most $AB+1$ (there is still the part $B_{i,i+2}$). Adding $T$, this gives $AB+2$ red edges incident to a vertex in $X$ (or from part $X_0$).
\item Add $x_{\ell+1}$ to $X_0$. The part $X_0$ can have red edges only to non-empty parts $B_{j,t}$  and to $T$, but no red edges to the single vertices. Thus, it has red degree at most $AB+2$.   
\item Put the part $B_{i,i+2}$ into $T$. This part is only adjacent to vertices up to $x_{\ell+2+C}$, and thus has at most $C+2$ red edges.
\end{enumerate}

Thus, at each point, the red degree is always at most $M=\max\{AB,C\}+2$.

The process ends at step $\ell=k-C-1$. Then, all the vertices not in $X$ are in some parts, and there are at most $AB+1$ such parts. On the other side of the bipartition, we have part $X_0$ and $C+1$ single vertices. Thus, the graph is bipartite with both sides of size at most $M$. One can contract each part independently to finish the contraction sequence.

To conclude, taking $C=d-2$ and $A=B=\lfloor \sqrt{d-2}\rfloor$, we have $M\leqslant d$ and $kAB2^C=\Theta(kd2^d)$. Notice that we may assume that $A,B$ and $C$ are positive since $d\geq d'$ where $d'$ was some well chosen positive constant. This concludes the proof.
\end{proof}

\section{Conclusion}\label{sec:conclu}

We have given an essentially tight upper bound for the neighbourhood complexity of graphs of bounded twin-width together with a construction almost attaining this upper bound. Moreover, our method is simple and self-contained. A similar upper bound was implied by the techniques in \cite{Gajarsky22} (though not stated explicitly). 

It is known that the twin-width of $G^r$ can be upper-bounded by a function of the twin-width of $G$ and~$r$~\cite{twin-width-I}.
Thus, graphs of twin-width at most~$d$ have linear $r$-neighbourhood complexity. Recently, improved bounds were given for planar graphs and proper minor-closed graph classes in~\cite{joret2023neighborhood} (such graphs also have bounded twin-width). 
We leave as an interesting open problem to obtain an essentially tight twin-width dependence for the $r$-neighbourhood complexity.

We remark that the neighbourhood complexity is also related to \emph{identification problems} on graphs such as \emph{identifying codes} or \emph{locating-dominating sets}, where one seeks a (small) set $A$ of vertices of a~graph such that all other vertices have a distinct neighbourhood in $A$~\cite{FMNPV17}. Some works in this area about specific graph classes, are equivalent to the study of the neighbourhood complexity of these graph classes: see for example~\cite{BLLPT15,FMNPV17,RS84}. Moreover, we note that for graph classes with VC density~1, since any solution has linear size, the natural minimisation versions of the above identification problems have a polynomial-time constant-factor approximation algorithm (trivially select the whole vertex set), while such an algorithm is unlikely to exist in the general case~\cite{BLLPT15}. Thus, the bounds given in the current work imply a better approximation ratio for these problems, when restricted to input graph classes of bounded twin-width.

\medskip

\textbf{Acknowledgement.}
We thank an anonymous reviewer for pointing out the implicit upper bound of~\cite{Gajarsky22} mentioned in the introduction.

\bibliographystyle{plain}
\bibliography{main}

\end{document}